\documentclass[1 [leqno,11pt]{amsart}
\usepackage{amssymb, amsmath}
\usepackage{enumerate}
\def\inte#1{
\displaystyle\mathop{#1\kern0pt}^\circ }

\let\pa=\partial
\let\p=\partial

\let\e=\varepsilon

\let\r=\rho
\let\s=\sigma
\let\f=\frac

\let\na=\nabla

\let\D=\Delta

\let\Om=\Omega

\def\bU{\mathbf{U}}
\def\bV{\mathbf{V}}

\def\cH{{\mathcal H}}

\def\cS{{\mathcal S}}

\def\cW{{\mathcal W}}
\def\cX{{\mathcal X}}

\def\grad{\nabla}
\def\la{\lambda}

\renewcommand{\div}{{\rm div}\,}
\newcommand{\loc}{{\rm loc}\,}
\newcommand{\Lip}{{\rm Lip}\,}

\newcommand\ds{\displaystyle}

\def\tp{p_1}

\def\dB{\dot{B}}

\def\virgp{\raise 2pt\hbox{,}}
\def\cdotpv{\raise 2pt\hbox{;}}

\def\eqdefa{\buildrel\hbox{\footnotesize def}\over =}

\def\C{\mathop{\mathbb C\kern 0pt}\nolimits}
\def\DD{\mathop{\mathbb D\kern 0pt}\nolimits}
\def\EE{\mathop{{\mathbb E \kern 0pt}}\nolimits}
\def\K{\mathop{\mathbb K\kern 0pt}\nolimits}
\def\N{\mathop{\mathbb N\kern 0pt}\nolimits}
\def\Q{\mathop{\mathbb Q\kern 0pt}\nolimits}
\def\R{\mathop{\mathbb R\kern 0pt}\nolimits}
\def\SS{\mathop{\mathbb S\kern 0pt}\nolimits}
\def\ZZ{\mathop{\mathbb Z\kern 0pt}\nolimits}
\def\TT{\mathop{\mathbb T\kern 0pt}\nolimits}
\def\PP{\mathop{\mathbb P\kern 0pt}\nolimits}

\newcommand{\Z}{{\ZZ}}

\def\dive{\mathop{\rm div}\nolimits}

\def\na{\nabla}
\def\d{\partial}

\newcommand{\beq}{\begin{equation}}
\newcommand{\eeq}{\end{equation}}
\newcommand{\ben}{\begin{eqnarray}}
\newcommand{\een}{\end{eqnarray}}
\newcommand{\beno}{\begin{eqnarray*}}
\newcommand{\eeno}{\end{eqnarray*}}
\newcommand{\andf}{\quad\hbox{and}\quad}
\newcommand{\with}{\quad\hbox{with}\quad}
\newtheorem{defi}{Definition}[section]
\newtheorem{thm}{Theorem}[section]
\newtheorem{lem}{Lemma}[section]
\newtheorem{rmk}{Remark}[section]

\newtheorem{prop}{Proposition}[section]

\begin{document}
\title[Global regularity of 3-D density patch]
{Global  regularity of three dimensional density patches for inhomogeneous
incompressible viscous flow}
 \author[X. Liao]{Xian Liao}
\address [X. Liao]%
{Mathematisches Institut, Universit\"at Bonn,
Endenicher Allee 60, 53115 Bonn, Germany.}
\email{xianliao@math.uni-bonn.de}
\author[Y. Liu]{Yanlin Liu}%
\address[Y. Liu]
 {Department of Mathematical Sciences, University of Science and Technology of China, Hefei 230026, China.}
\email{liuyanlin3.14@126.com}

\date{\today}
\maketitle

\begin{abstract}
Toward P.-L. Lions' open question  in \cite{Lions96} concerning the propagation of
regularity for density patch, we prove that the boundary regularity of the 3-D density patch
persists by time evolution for inhomogeneous incompressible viscous flow, with the initial density
given by $(1-\eta){\bf 1}_{\Om_0}+{\bf 1}_{\Om_0^c}$
for some small enough constant $\eta$ and some $W^{k+2,p}$ domain $\Om_0$, $p\in]3,\infty[$,
and with the initial velocity satisfying some smallness condition
and appropriate conormal regularities.
\end{abstract}

\noindent {Inhomogeneous incompressible Navier-Stokes equations; density patch;
striated distribution spaces}


 \setcounter{equation}{0}
\section{Introduction}

We  consider the following inhomogeneous incompressible 3-D Navier-Stokes equations:
\begin{equation}\label{1.1}
 \left\{\begin{array}{l}
\displaystyle \pa_t\rho+\div(\rho v)=0,\qquad (t,x)\in\R^+\times\R^3,\\
\displaystyle\rho(\pa_t v+v\cdot\nabla v)-\Delta v+\nabla\pi=0,\\
\displaystyle \div v=0,\\
\displaystyle  (\rho, v)|_{t=0}=(\rho_0, v_0),
\end{array}\right.
\end{equation}
where $\rho\in\R^+,$ $v\in\R^3$ and $\pi\in\R$ stand for the density, velocity field
and pressure of the fluid respectively.
System \eqref{1.1} describes an incompressible fluid
with variable density.
Basic examples are mixture of incompressible and non reactant flows,
models of rivers, fluids containing a melted substance, etc.

The existence and uniqueness issues of the solutions of \eqref{1.1} have been studied extensively.
We just cite here among many others
\cite{AKM,Lions96, Simon}  for the construction of the weak solutions of \eqref{1.1}, and the works \cite{AP,AKM,CK,D1, LS}  for  the wellposedness results for the strong solutions.
Recently progresses have been made:
the smallness conditions on the density fluctuation was successfully removed in e.g. \cite{AGZ2, Danchin04'}, a small jump of the density function across some interface was permitted in e.g. \cite{DM12, HPZ}, and in the  energy framework, \cite{PZZ} considered the positive density which is only assumed to be bounded from up and below.

In \cite{Lions96} P.-L. Lions proposed the following density patch problem:
if the initial density  $\rho_0=\mathbf{1}_D$ for some smooth domain $D$, then whether or not the boundary regularity of $D$ will persist by time evolution?
There are many recent progresses toward this problem: Danchin-Mucha
\cite{DM17} propagated the $C^{1,\alpha}$ boundary regularity in the presence of vacuum. See \cite{Gancedo, DZ} for the  propagation results of the lower-order H\"older boundary  regularity. In  \cite{LZ, LZ2}, the first author and P. Zhang considered this  density patch problem with the boundary of any high-order regularity, in space dimension two, away from vacuum.
More precisely in \cite{LZ},   the initial density $\rho_0$ is taken of the following form
\begin{equation}\label{rho0}
{\rho_0}={(1-\eta)\bf 1}_{\Om_0}+ {\bf 1}_{\Om_0^c},
\quad 1-\eta\in\R^+,
\end{equation}
where $|\eta|$ is a sufficiently small constant, and $\Omega_0$ is some
simply connected $W^{k+2,p}(\R^2)$,  $k\geq 1,$ $p\in ]2,4[$ bounded domain in $\R^2$.
Let $X_0 \in W^{k+1,p}(\R^2)$
 be the divergence-free tangential vector field of $\d\Omega_0$
\footnote{Let $g_0\in W^{k+2,p}(\R^2)$ such that $\pa \Omega_0=g_0^{-1}(0)$ and $\nabla g_0$ does not vanish on $\pa \Omega_0$.
Then we can choose $X_0=\nabla^\bot g_0$.}
and denote by
 $\p_{X_0}f\eqdefa X_0\cdot\na f=\div(fX_0)$ the derivative of $f$ in the direction of $X_0$.
Then   they proved the following persistence result of the boundary regularity:
\begin{thm}\label{thm1.1}
{\sl
Given initial density $\rho_0$ of \eqref{rho0}
with $|\eta|$ sufficiently small, and initial velocity $v_0$ of \eqref{1.1}
satisfies the following conormal regularities for some $\epsilon\in ]0,1[$:
\begin{equation}\label{LZv0}
\begin{split}
v_0\in W^{1,p}(\R^2) \with \p_{X_0}^\ell v_0\in
 W^{1-\f{\ell}k\epsilon,p}(\R^2),
 \quad \ell=1,\cdots, k.
    \end{split}\end{equation}
Then the Cauchy problem \eqref{1.1}-\eqref{rho0}-\eqref{LZv0} has a unique global solution $(\r, v)$ such that
$$\r(t,x)={(1-\eta)\bf 1}_{\Om(t)}+{\bf
1}_{\Om(t)^c},\quad\mbox{with}\quad \Om(t)=\psi(t,\Omega_0),$$
where $\psi$ is the flow of $v$, and $\Om(t)$ remains a simply connected bounded $W^{k+2,p}(\R^2)$ domain.
}
\end{thm}
The smallness condition on the jump $\eta$ was successfully removed in \cite{LZ2},
by using time weighted energy estimates.
There, the initial velocity $v_0$, together with its conormal derivatives $\d_{X_0}^\ell v_0$, is taken in the energy spaces.
Precisely, there exists some $s_0\in ]0,1[$ and some
$\epsilon_1\in ]0,s_0[$, such that
$$
v_0\in L^2\cap\dB^{s_0}_{2,1},
\quad
  \d_{X_0}^\ell v_0\in
L^2\cap\dB^{s_\ell}_{2,1}\quad\mbox{for}\quad  s_\ell\eqdefa s_0-\epsilon_1{ \ell}/{k},
\ \ell=1,\cdots,k.$$

The purpose of this paper is to extend Theorem \ref{thm1.1} to three dimensional case,
for the density patch problem with small jump \eqref{rho0} but not in the finite-energy framework.

One of the main tools we use in this paper is the so-called conormal (or striated) distribution spaces,
which have been successfully used by J.-Y. Chemin \cite{Chemin91, Chemin93} for studying the
evolution of the boundary regularity  of the 2-D vortex patch problem for Euler equations (see also \cite{BC93}).
Then the subsequent works \cite{D97} considered the viscous case
and the works \cite{GR,ZQ} extended the results to the 3-D case.

There are two obvious difficulties when considering 3-D case.
One comes from the fact that,
unlike 2-D case, the global well-posedness of 3-D Navier-Stokes equations
remains still unknown to us. In fact, it is one of the most challenging open problems in fluid mechanics.
So here we choose to
assume the following smallness condition on the initial velocity
\begin{equation}\label{HPZv0}
\begin{split}
 & \|v_0 \|_{\dB^{-1+\frac 3\tp}_{\tp,r}(\R^3)} \leq c_0,
 \hbox{ for some } 1<p_1<3
 \hbox{ and } 1<r<\min\{\frac{2\tp}{3(\tp-1)},\, \frac{4p}{2p-3},\,3 \},
    \end{split}
 \end{equation}
where $c_0$ is some sufficiently small constant,
and $\dB^{-1+\frac 3\tp}_{\tp,r}$ is
Besov space, see Definition \ref{defBesov} below.
Under this smallness condition on the initial data,  we have the following proposition
to guarantee the existence of  global weak solutions to \eqref{1.1}.

\begin{prop}[a particular case of Theorem $1.1$ in \cite{HPZ}]\label{thmexistence}
{\sl Let the initial data $(\rho_0, v_0)$ satisfy \eqref{rho0} and \eqref{HPZv0}.
 If $|\eta|, c_0$ are sufficiently small, then \eqref{1.1}  has a  global  weak solution $(\rho, v)$, and there exists a positive constant $C$  such that
\begin{equation}\label{vLr}\begin{split}
\|(\Delta v, \nabla\pi)\|_{L^r(\R^+;L^{\frac{3r}{3r-2}}(\R^3))}&+\|\nabla v\|_{L^{2r}(\R^+;L^{\frac{3r}{2r-1}}(\R^3))}
\leq C\|v_0\|_{\dB^{-1+\frac 3\tp}_{\tp,r}(\R^3)}.
\end{split}\end{equation}
}
\end{prop}

The other difficulty is that,
the boundary of a two dimensional domain is one dimensional curve, whose tangent space
can be spanned by the tangent vector (e.g. by $X_0$ given above).
But the boundary
of a three dimensional domain is a two dimensional surface, whose tangent space has dimension two,
hence in order to propagate the boundary regularity,
we need to discuss differentiations in different directions, which makes the problem more complicated.

To deal with this difficulty, we here follow \cite{GR} to  select ``good'' tangent vector directions to work with.
We first adopt the following definition of admissible systems introduced in \cite{GR}:

\begin{defi}[Definition 3.1 of \cite{GR}]\label{def1.1}
{\sl Any system $\cW=\{W_1,\cdots,W_N\}$ composed of N continuous vector fields is said to be \emph{admissible} if the function
$$[\cW]^{-1}\eqdefa
\Bigl( \frac{2}{N(N-1)}\sum_{\mu<\nu}|W_\mu\wedge W_\nu|^2 \Bigr)^{-\frac 14}$$
is bounded.
Here, for any two vector fields $X=(X^1, X^2, X^3)^T$, $Y=(Y^1, Y^2, Y^3)^T$, the wedge product $X\wedge Y$ is defined as
$X\wedge Y=\left(
X^2 Y^3-X^3 Y^2,
X^3Y^1-X^1 Y^3,
X^1Y^2-X^2 Y^1
\right)^T.$
}\end{defi}

Now we can select five \emph{divergence-free} tangential vector fields for the two dimensional boundary according to the following result,
which is Proposition 3.2 of \cite{GR}:
\footnote{Proposition 3.2 in \cite{GR} concerns the framework of H\"older spaces, however the proof also works in the framework of Sobolev spaces.}
\begin{prop}\label{prop1.1}
For any $W^{k+2,p}(\R^3)$
two dimensional compact submanifold $\Sigma$ of $\R^3$,
we can find an admissible system
consisting of five $W^{k+1,p}(\R^3)$, divergence-free vector fields tangent to $\Sigma$.
\end{prop}

Before we go to the statement of density patch problem in three space dimension, let us introduce the following notations which will be used freely in the following context.
For any system $\cW=\{W_1,\cdots,W_N\}$
composed of $N$ continuous vector fields,
and for any multi-index $\alpha=(\alpha_1,\cdots,\alpha_m)$ of length $m$ with $\alpha_1,\cdots,\alpha_m\in\{1,\cdots,N\}$, we denote
\begin{equation*}\label{1.3}
\pa_{\cW}^\alpha=\pa_{\cW}^{(\alpha_1,\cdots,\alpha_m)}
=\pa_{W_{\alpha_1}}\cdots\pa_{W_{\alpha_m}}.
\end{equation*}
We emphasize that the order of differentiation is important.
We furthermore denote
\begin{equation*}\label{1.4}
\pa_{\cW}^m=\bigl(\pa_{\cW}^\alpha\bigr)_\alpha
\end{equation*}
to be an $N^m$ dimensional vector, where $\alpha$ takes over all the  multi-index with length $m$ and with elements taking integer  values between $1$ and $N$.

For a Banach space $B$, we shall use the shorthand $\|u\|_{L^p_TB}$ for the norm $\bigl\|\|u(t,\cdot)\|_B\bigr\|_{L^p(0,T)}$, and we will simply write $v\in B$ for
  the vector-valued field $v\in B^3$.
The notation $C$ stands for some real positive constant which may be different in each occurrence.

Now we can state   the density patch problem in $\R^3$.
Let  us take  the initial density $\rho_0$  of the form \eqref{rho0} where
 $\Omega_0$ is a simply connected bounded domain of $\R^3$,
such that its boundary $\d\Omega_0$ is $W^{k+2,p}(\R^3),$
  $k\geq 1,$ $p\in ]3,\infty[,$ two dimensional compact submanifold in $\R^3$.
By Proposition \ref{prop1.1}, we can find an admissible system
as follows
\begin{equation}\label{W0}
\begin{split}
&\cX_0=\{X_{1,0},X_{2,0},X_{3,0},X_{4,0},X_{5,0}\},
\quad\hbox{with $X_{i,0}$ tangent to $\pa\Omega_0$, and}\\
& X_{i,0}\in W^{k+1,p}(\R^3),
\,k\geq 1, \,p\in ]3,\infty[,
\quad\div X_{i,0}=0,
\quad i=1,\cdots,5.
\end{split}
\end{equation}
We assume the following (conormal) regularities on $v_0$ for some $\varepsilon\in ]0,1[$ \begin{equation}\label{v0}
\begin{split}
&v_0\in W^{1,p}(\R^3) \cap L^3(\R^3),
\quad\p_{\cX_0}^\ell v_0\in
 W^{1-\f{\ell}k\varepsilon,p}(\R^3),\quad \ell=1,\cdots, k.
    \end{split}
 \end{equation}

In the following, we consider the propagation
of the boundary regularity for the 3-D density patch problem.
The main result of this paper states as follows:
\begin{thm}\label{thmmain}
{\sl Let the initial data $(\rho_0, v_0)$ satisfy \eqref{rho0}-\eqref{HPZv0}-\eqref{W0}-\eqref{v0}, with $|\eta|, c_0$ sufficiently small.
 Then \eqref{1.1} has a unique global
solution $(\r, v)$ such that
$$\r(t,x)={(1-\eta)\bf 1}_{\Om(t)}+{\bf
1}_{\Om(t)^c},\quad\mbox{with}\quad \Om(t)=\psi(t,\Omega_0),$$
where $\psi$ is the flow of $v$, and $\Om(t)$ remains a simply connected bounded $W^{k+2,p}(\R^3)$ domain.
}\end{thm}

\begin{rmk}
{\sl Results in the same spirit as Theorem \ref{thm1.1} and Theorem \ref{thmmain},
have been obtained recently by R. Danchin and X. Zhang, see \cite{DZ}.
Their strategy is to assume that $\rho$ belongs to some multiplier space,
which allows them to propagate lower-order H\"older regularity.
}\end{rmk}

Recall in  \cite{LZ} that in the two dimensional case, if the velocity of the flow $v\in L^1_\loc(\R^+, W^{2,p})$, then  the tangential vector field $X$  is transported by the flow from $X_0$ as follows
$$
\d_t X+v\cdot\nabla X=X\cdot\nabla v,
\quad X|_{t=0}=X_0,
$$
and  to prove the persistence of  $W^{k+2,p}(\R^2)$ regularity of the domain $\Omega_0$ is equivalent  to show that the tangential vector field $X$     satisfies the following (conormal) regularities:
$$
X,\,\d_X X, \cdots,\d_X^{k-1}X\in W^{2,p}(\R^2),
\ \mbox{where}\ \div X=0
\hbox{ and } \d_X f\eqdefa X\cdot\nabla f=\div(fX).
$$

Here we consider the system $\cX(t)=\{X_1(t),\cdots,X_5(t)\}$ satisfying
for any $1\leq i\leq 5$,
 \beq\label{W}
 \left\{\begin{array}{l}
\displaystyle \pa_t X_i+v\cdot\nabla X_i=X_i\cdot\na v,\\
\displaystyle X_i(0,x)=X_{i,0}(x),
\end{array}\right. \eeq
where the initial system $\cX_0=\{X_{0,1},\cdots, X_{0,5}\}$ is given by \eqref{W0},
and we aim to control $\|(\d_{\cX}^\ell X_i)(t,\cdot)\|_{W^{2,p}}$.
The above choice of $\cX$ ensures that the directional derivatives  $\d_{X_i}$ and the material derivative
 $(\d_t+v\cdot\nabla)$ commute with each other, and hence by view of \eqref{1.1}, $\d_{X_i}^\ell \rho$ satisfy the free transport equations
$$
(\d_t+v\cdot\nabla) (\d_{X_i}^\ell\rho)=0, \quad \forall\ell=1,\cdots,k,
\quad \forall i=1,\cdots,5,
$$
with the initial condition $\d_{X_{i,0}}^\ell \rho_0\equiv 0$,
which is guaranteed by the choice of $\rho_0$ and $\cX_0$, see \eqref{rho0} and \eqref{W0}.
Therefore we can derive that
\begin{equation}\label{bound:a}
\d_{\cX}^\ell \rho\equiv0,
\quad \forall \ell=1,\cdots,k,
\end{equation}
 which means that the tangential regularity for $\rho$  propagates
for all  the times.
On the other side, $X_i(t)$ should remain divergence-free all the time
since $\div X_i$ satisfies the free transport equation by \eqref{W}:
\begin{equation*}
 \left\{\begin{array}{l}
\displaystyle \pa_t (\div X_i)+v\cdot\nabla (\div X_i)=0,\\
\displaystyle \div X_i|_{t=0}=\div X_{i,0}=0.
\end{array}\right.
\end{equation*}

Moreover, we have the following proposition:
\begin{prop}\label{propapriori}
{\sl Under the assumptions of  Theorem \ref{thmmain}, the system \eqref{1.1} and \eqref{W}
has a unique global solution $(v,\rho,\cX)$ satisfying for any $i\in\{1,\cdots,5\},~\ell\in \{1,\cdots, k\}$ that
\begin{equation}\label{claim}
v\in L^1_\loc(\R^+, W^{2,p}(\R^3)),
\quad\pa_{\cX}^{\ell-1}X_i\in L^\infty_\loc(\R^+; W^{2,p}(\R^3)),\quad
\dive X_i=0.
\end{equation}
}\end{prop}

The proof of this proposition will be postponed to the next section.
In the following, we shall present the proof of Theorem \ref{thmmain}
by assuming Proposition \ref{propapriori}.

\begin{proof}[Proof of Theorem \ref{thmmain}]
For the unique global solution $(v,\rho)$ of \eqref{1.1}
given by Proposition \ref{propapriori},
let $\psi(t,x)$ be the flow associated with the velocity field $v$, that is,
\begin{equation}\label{psi}
 \left\{\begin{array}{l}
\displaystyle \pa_t\psi(t,x)=v(t,\psi(t,x)),\\
\displaystyle \psi(0,x))=x,
\end{array}\right.
\end{equation}
then by view of \eqref{claim},   we have
\begin{equation*}
\psi(t,\cdot)-Id\in L^\infty_{\loc}(\R^+;W^{2,p}),
\end{equation*}
and hence $\Omega(t)\eqdefa\psi(t,\Omega_0)$ is a $W^{2,p}(\R^3)$ domain. Moreover,
the first equation of \eqref{1.1} gives
\begin{equation}\label{rho}
\rho(t,x)={(1-\eta)\bf 1}_{\Om(t)}+{\bf 1}_{\Om(t)^c},
\quad 1-\eta\in \R^+.
\end{equation}

On the other side, due to the Finite Covering Theorem,
there exists a finite number of charts   $\{\bV^\beta\}_{\beta=1}^m$ covering  the two dimensional $W^{k+2,p}(\R^3)$ compact submanifold $\pa\Omega_0$
such that we can parametrize anyone of them, say $\bV^1$,  as follows
 \begin{equation*}\label{1.6}
\phi_1: \bU^1\rightarrow   \bV^1,
\quad\hbox{via}\quad (r,s)\mapsto\phi_1(r,s),
\quad \phi_1\in W^{k+2,p}(\bU^1),
 \end{equation*}
 where $\bU^1$ is an open set on $\R^2$ and $\bV^1$ is $\d\Omega_0$-open set in $\R^3$.
In order to show that $\pa\Omega(t)=\psi(t,\pa\Omega_0)\in W^{k+2,p}$,  $k\geq 1$,
it suffices to show (without loss of generality)
\begin{equation*}
\pa_r^{k_1}\pa_s^{k_2}\psi(t,\phi_1(r,s))\in L^\infty_\loc(W^{2,p}(\bU^1)),\ \forall k_1+k_2=k.
\end{equation*}
Hence we only need to verify that
\begin{equation}\label{YZ}
(\pa_{Y_0}^{k_1}\pa_{Z_0}^{k_2}\psi)(t,\cdot)\in L^\infty_\loc(W^{2,p}(\bV^1)),\ \forall k_1+k_2=k,
\end{equation}
where the tangent vector fields $Y_0, Z_0\in W^{k+1,p}(\bV^1;\R^3)$ are defined by
 $$Y_0(\phi_1(r,s))=\pa_r\phi_1(r,s),
 \quad
 Z_0(\phi_1(r,s))=\pa_s\phi_1(r,s).
 $$
Indeed,   a direct calculation gives for any $(r,s)\in \bU^1$,
\begin{align*}
&\pa_r\psi(t,\phi_1(r,s))
=Y_0(\phi_1(r,s))\frac{\pa\psi(t,\phi_1(r,s))}{\pa x}
=(\pa_{Y_0}\psi)(t,\phi_1(r,s)),
\\
& \pa_s\psi(t,\phi_1(r,s))
=Z_0(\phi_1(r,s))\frac{\pa\psi(t,\phi_1(r,s))}{\pa x}
=(\pa_{Z_0}\psi)(t,\phi_1(r,s)),
\end{align*}
and hence by an induction argument we achieve
\begin{equation*}\label{5.6}
\pa_r^{k_1}\pa_s^{k_2}\psi(t,\phi_1(r,s))
=(\pa_{Y_0}^{k_1}\pa_{Z_0}^{k_2}\psi)(t,\phi_1(r,s)),
\hbox{ with }\phi_1\in W^{k+2,p}(\bU^1).
\end{equation*}

Noting that the initial system $\cX_0$ given in \eqref{W0} is admissible,
so for any $(\widetilde r, \widetilde s)\in\bU^1$,
without loss of generality, we can assume
\begin{equation}\label{5.8}
|X_{1,0}\wedge X_{2,0}|(\phi_1(\widetilde r, \widetilde s))
\geq \min_{(\widetilde r, \widetilde s)\in \bU^1}\bigl([\cX_0](\phi_1(\widetilde r, \widetilde s))\bigr)^2>0.
\end{equation}
Then the fact that $X_{i,0}$ is continuous since $W^{k+1,p}(\R^3)\hookrightarrow C(\R^3)$,
guarantees the existence of
a $\d\Omega_0$-open set $\bV_0\subset \bV^1$ containing $\phi_1(\tilde{r},\tilde{s})$ such that
\begin{equation}\label{5.9}
\inf_{x\in\bV_0 }|X_{1,0}\wedge X_{2,0}|(x) >0.
\end{equation}
Thus we can decompose $Y_0$, $Z_0$ as a linear combination of $X_{1,0}$ and $X_{2,0}$ on $\bV_0$, namely
\begin{equation*}\label{5.10}
Y_0=c_1 X_{1,0}+c_2 X_{2,0},\quad Z_0=d_1 X_{1,0}+d_2 X_{2,0},
\end{equation*}
where the coefficients are defined by
\begin{equation*}\label{5.11}\begin{split}
&c_i=\frac{(Y_0,X_{i,0})|X_{j,0}|^2-(Y_0,X_{j,0})(X_{1,0},X_{2,0})}{|X_{1,0}\wedge X_{2,0}|^2},\\
&d_i=\frac{(Z_0,X_{i,0})|X_{j,0}|^2-(Z_0,X_{j,0})(X_{1,0},X_{2,0})}{|X_{1,0}\wedge X_{2,0}|^2},
\quad i,j=1,2,\, i\neq j.
\end{split}\end{equation*}
Hence without loss of generality,  to prove \eqref{YZ}  reduces to prove
\begin{equation}\label{5.14}
(\pa_{c_1 X_{1,0}+c_2 X_{2,0}}^{k_1}\pa_{d_1 X_{1,0}+d_2 X_{2,0}}^{k_2}\psi)(t,\cdot)\in L^\infty_\loc(W^{2,p}(\bV_0)),\ \forall k_1+k_2=k.
\end{equation}
A new problem needed to consider is that the differential may act on the coefficients $c_1,c_2,d_1,d_2$.
But thanks to the fact $X_{1,0}, X_{2,0}, Y_0, Z_0\in W^{k+1,p}(\bV^1)$, $k\geq 1$, $p>3$,
and \eqref{5.9},
we know that the coefficients $c_i, d_i$, $i=1,2$ belong to $W^{k+1,p}(\bV_0)$, $k\geq 1$.
Hence it suffices to show
\begin{equation}\label{5.15}
\pa_{(X_{1,0},X_{2,0})}^{k}\psi\in L^\infty_\loc(W^{2,p}(\bV_0)).
\end{equation}
To do this, let us recall the definition \eqref{psi} of the stream function $\psi$,
thus the vector field $X_i(t,x)$ defined by \eqref{W} can be written as
\begin{equation}\label{5.17}
X_i(t,x)=(\pa_{X_{i,0}}\psi)(t,\psi^{-1}(t,x)),
\quad   i=1,2.
\end{equation}
Hence for any function $f(t,x)\equiv g(t,\psi^{-1}(t,x))$, there holds
\begin{equation*}\label{5.18}\begin{split}
(\pa_{X_i}&f)(t,x)=(\pa_{X_{i,0}}\psi)(t,\psi^{-1}(t,x))
\cdot\nabla\psi^{-1}(t,x)\cdot(\nabla g )(t,\psi^{-1}(t,x))\\
&=X_{i,0}(t,\psi^{-1}(t,x))\cdot(\nabla\psi)(t,\psi^{-1}(t,x))
\cdot\nabla\psi^{-1}(t,x)\cdot(\nabla g )(t,\psi^{-1}(t,x))\\
&=X_{i,0}(t,\psi^{-1}(t,x))\cdot(\nabla g )(t,\psi^{-1}(t,x))\\
&=(\pa_{X_{i,0}} g)(t,\psi^{-1}(t,x)),
\quad i=1,2.
\end{split}\end{equation*}
Applying the above formula repeatedly  on \eqref{5.17} yields that
\begin{equation*}\label{5.19}\begin{split}
(\pa_{(X_1,X_2)}^{\alpha(k-1)}X_{\alpha_k})(t,x)
= (\pa_{(X_{1,0},X_{2,0})}^{\alpha(k)}\psi)\,(t,\psi^{-1}(t,x)),
\end{split}\end{equation*}
for any multi-index $\alpha(k)=(\alpha_1,\cdots,\alpha_k)$ with $\alpha_1,\cdots,\alpha_k\in\{1,2\}$,
and $\alpha(k-1)$ denotes $(\alpha_1,\cdots,\alpha_{k-1})$.
Hence in order to prove \eqref{5.15}, we only need to show
\begin{equation}\label{5.16}
\pa_{(X_{1},X_{2})}^{k-1}(X_{1},X_{2})\in L^\infty_\loc\bigl(W^{2,p}(\mathbf{V}(t))\bigr),
\quad \mathbf{V}(t)\eqdefa\psi(t,\bV_0),
\end{equation}
which is guaranteed by \eqref{claim}.
This completes the proof of this theorem.
\end{proof}

At the end of this section, we recall some basic facts on
Littlewood-Paley theory for readers' convenience.
For $u\in{\mathcal S}'(\R^d),$
where ${\mathcal S}'$ stands for tempered distribution space, we set
$$
\dot\Delta_j u=\varphi(2^{-j}D)u\quad\mbox{and} \quad \dot{S}_j u=\chi(2^{-j}D)u,\quad \forall j\in\Z,
$$
where $\chi,\varphi$ are smooth
functions  supported in the ball $\mathcal{B}\eqdefa \{ \xi\in\R^d,
|\xi|\leq\frac{4}{3}\}$ and the ring $\mathcal{C}\eqdefa \{
\xi\in\R^d,\frac{3}{4}\leq|\xi|\leq\frac{8}{3}\}$ respectively, such that
\begin{equation*}
 \sum_{j\in\Z}\varphi(2^{-j}\xi)=1 \quad\hbox{for}\quad \xi\neq 0,\andf \chi(\xi)+\sum_{j\geq 0}
 \varphi(2^{-j}\xi)=1 \quad\hbox{for}\quad \xi\in\R^d.
\end{equation*}
Then we can define the homogeneous Besov space $\dot B^s_{p,r}$ with $s<\frac{d}{p}$ as follows.
\begin{defi}\label{defBesov}
{\sl  Let $(p,r)\in[1,+\infty]^2,~s<\frac{d}{p}$  and $u\in{\mathcal
S}_h'(\R^d),$ which means that $u$ is in~$\cS'(\R^d)$ and
satisfies~$\ds\lim_{j\to-\infty}\|\dot S_ju\|_{L^\infty}=0$. We define
$$\dot B^s_{p,r}(\R^d)\eqdefa \big\{u\in{\mathcal S}_h'(\R^d)\;\big|\; \|u\|_{\dot B^s_{p,r}}<\infty\big\},
\quad\mbox{where}\quad
\|u\|_{\dot B^s_{p,r}}\eqdefa \Big\| \Bigl(2^{js}\|\dot\Delta_j
u\|_{L^{p}}\Bigr)_{j\in\Z}\Big\|_{\ell ^{r}}.
$$
}\end{defi}

\setcounter{equation}{0}
\section{The proof of Proposition \ref{propapriori}}

The purpose of this section is to present the proof of Proposition \ref{propapriori}.
For any smooth enough solution $(v,\rho,\cX)$ of \eqref{1.1} and \eqref{W}, we consider
\begin{equation}\label{J0}
J_0(t)\eqdefa 1+\|(\d_t v,\Delta v,\nabla\pi)\|_{L^{r_0}_t(L^p)}
+\|\nabla v\|_{L^{\sigma_1}_t(L^\infty)\cap L^{s_0}_t(L^p)}
 +\|v\|_{L^\infty_t(L^3)\cap L^{\sigma_2}_t(L^\infty)} ,
\end{equation}
and the following inductively defined quantities $J_\ell(t)$,  $ \ell=1,\cdots,k,$
\begin{equation}\label{Jell}
\begin{split}
J_{\ell}(t)\eqdefa
&J_{\ell-1}(t)
 + \|(\d_t \d_{\cX}^{\ell } v,\Delta\d_{\cX}^{\ell} v, \nabla\d_{\cX}^{\ell
}\pi)\|_{L^{r_{\ell }}_t(L^p)}
+\|\nabla\d_{\cX}^{\ell } v\|_{L^{r_{\ell }}_t(L^\infty) \cap
L^{s_{\ell }}_t(L^p)} \\
&+\|\d_{\cX}^{\ell } v\|_{L^{s_{\ell }}_t(L^\infty)\cap L^\infty_t(L^p)}
+ \|\d_t \d_{\cX}^{\ell-1}{\cX}\|_{L^{s_{\ell }}_t(W^{1,p})} +
\|\d_{\cX}^{\ell-1} {\cX}\|_{L^\infty_t(W^{2,p})},
\end{split} \end{equation}
where $\cX$ is defined by \eqref{W}, and $r_0,~s_0,~\sigma_1,~\sigma_2,~r_\ell,~s_\ell$ can be taken freely as long as
\begin{equation}\begin{split}\label{r}
&r_0\in]1,2[,\quad s_0\in]2,\infty[,\quad
\sigma_1\in
\bigl]\f{2p}{p+3}, \frac {2p}{3}\bigr[,
\quad\sigma_2\in\bigl]\frac{4p-6}{p}, \infty[,\quad\mbox{and}\\
&\qquad\qquad\quad r_\ell \in
\bigl]1,\f{2k}{k+\ell\e}\bigr[,
\quad s_\ell\in
\bigl]2,\f{2k}{\ell\e}\bigr[,
\quad \ell=1,\cdots,k.
\end{split}\end{equation}
Then the main ingredient of proving Proposition \ref{propapriori} is the following {\it A priori} estimates:
\begin{equation}\label{bound}
J_0(t)\leq C_0,\quad
J_\ell(t)\leq
\cH_{\ell+1}(t)\eqdefa C_0\underbrace{\exp\cdots\exp}_{\ell+1\hbox{ times }}
(C_0 t),
\quad
\forall \ell=1,\cdots,k,
\quad \forall t\in\R^+.
\end{equation}
Here and in all that follows, $C_0$ denotes some positive constant
which depends only on the initial data and may vary from lines to lines in the following context.
We remark that \eqref{bound} is not the explicit bound, the huge number of exponentials is only a technical artefact.

We shall prove \eqref{bound} in the following Subsections \ref{sub2.1},~\ref{sub2.2} and \ref{sub2.3}
for the case $\ell=0,~\ell=1$ and $\ell\geq 2$ respectively. And Subsection \ref{sub2.4}
is devoted to the proof of Proposition \ref{propapriori}.


\subsection{The proof of \eqref{bound} for $\ell=0$.}\label{sub2.1}

\ Let us first state the following {\it a priori} estimate.

 \begin{lem}\label{prop2.1}
{\sl Assume  the same hypothesis as in Theorem \ref{thmmain}.
Then for the global weak solution given in Proposition \ref{thmexistence}, there exists a positive constant $C_0$ such that
\beq\label{vL3}
\|v(t)\|_{L^3(\R^3)}^3+\|\nabla |v|^{\f32}\|_{L^2_t(L^2(\R^3))}^2\leq
C_0(\|v_0\|_{L^3(\R^3)}^3+1),
\quad\forall t\in\R^+.
\eeq
}
\end{lem}

\begin{proof}
It suffices to prove \eqref{vL3} for smooth solutions of the equation \eqref{1.1}
\footnote{Indeed, we can take approximated smooth solutions $(\rho_n, v_n, \pi_n)$ of the equation \eqref{1.1} accompanied with the mollified initial data $(1+S_n(\rho_0-1), S_nv_0)$ such that \eqref{vL3} holds uniformly in $n$. Then a passage to the limit implies Lemma \ref{prop2.1}.}.
Taking $L^2(\R^3)$ inner product between the momentum equation in \eqref{1.1} with $v|v|$ gives
\beq\label{2.6}
\f13\f{d}{dt}\|\rho^{\f13}v(t)\|_{L^3}^3
-\int_{\R^3}\Delta v\cdot v|v|dx
+\int_{\R^3}\nabla\pi\cdot v|v|dx=0.
\eeq
Integrating by parts, we have
\beq\label{2.7}
\begin{split}
-\int_{\R^3}\Delta v\cdot v|v|dx
&=\int_{\R^3} |\nabla v|^2 |v|dx
+\int_{\R^3} \nabla |v|\cdot\nabla v\cdot v dx\\
&\geq \frac12\int_{\R^3} \nabla|v|\cdot \nabla|v|^2 dx
=\Bigl\|\frac 23\nabla |v|^{\f32}\Bigr\|_{L^2}^2.
\end{split}
\eeq
Applying the three dimensional interpolation inequality that
$$\|f\|_{L^q(\R^3)}
\leq C\|f\|_{L^2(\R^3)}^{\f3q-\frac 12}\|\nabla f\|_{L^2(\R^3)}^{\frac 32-\f3q}
\quad \mbox{for}\quad \forall q\in[2,6],$$
and Young's inequality, we obtain for $r\in [1,3]$
\beq\label{2.8}
\begin{split}
\int_{\R^3}\nabla\pi\cdot v|v|dx
&\leq\|\nabla\pi\|_{L^{\f{3r}{3r-2}}}\bigl\||v|^{\f32}\bigr\|_{L^{2r}}^{\f43}\\
&\leq\|\nabla\pi\|_{L^{\f{3r}{3r-2}}}\bigl\||v|^{\f32}\bigr\|_{L^2}^{2\cdot\f{3-r}{3r}}
\bigl\|\nabla |v|^{\f32}\bigr\|_{L^2}^{2\cdot\f{r-1}{r}}\\
&\leq\frac 19\bigl\|\nabla |v|^{\f32}\bigr\|_{L^2}^2
+C\|\nabla\pi\|_{L^{\f{3r}{3r-2}}}^{r}\bigl\||v|^{\f32}\bigr\|_{L^2}^2+C\|\nabla\pi\|_{L^{\f{3r}{3r-2}}}^{r}.
\end{split}\eeq
Substituting \eqref{2.7}, \eqref{2.8} into \eqref{2.6}, we get
$$
\f13\f{d}{dt}\bigl\|\rho^{\f13}v(t)\bigr\|_{L^3}^3+\frac 13\bigl\|\nabla |v|^{\f32}\bigr\|_{L^2}^2\leq
C\|\nabla\pi\|_{L^{\f{3r}{3r-2}}}^{r}\|v\|_{L^3}^3+C\|\nabla\pi\|_{L^{\f{3r}{3r-2}}}^{r}.
$$
Then we use \eqref{rho}, Gronwall's inequality, and the estimate \eqref{vLr}
to achieve \eqref{vL3}.
\end{proof}

We will also use the following two lemmas,
which can be proved exactly along the same line of the proofs of Lemma 4.1, Lemma 4.2 in
\cite{LZ}, and Lemma 7.3 in \cite{LPG}.
Hence we just sketch their proofs for the readers' convenience.

\begin{lem}\label{lem3.1}
{\sl Let $p\in [\frac 32,\infty[$, $r\in ]1,2[$, $s\in ]2, \infty]$ and $q\in
\bigl]\f{4p}{2p+3}, \frac{4p}{3}\bigr[.$ Let $v_0\in W^{1,p}$ and $
v_L(t)\eqdefa e^{t\D}v_0.$ Then there exists some positive constant $C$ such that
\beq\label{3.1}
\|\D v_L\|_{L^{r}_t(L^p)}+\bigl\|\na v_L\bigr\|_{L^{s}_t(L^p)}+\|\na
v_L\|_{L^q_t(L^{2p})}\leq C\|v_0\|_{W^{1,p}},
\quad \forall t\in \R^+.
\eeq}
\end{lem}
\begin{proof}
The fact that $v_0\in W^{1,p}$ ensures
$$
\nabla^2 v_0\in \dot B^{-1}_{p,\infty}\cap \dot B^{-2}_{p,\infty}
\hookrightarrow \dot B^{-1-\sigma}_{p,1},
 \quad \nabla v_0\in \dot B^{0}_{p,\infty}\cap \dot B^{-1}_{p,\infty}
\hookrightarrow \dot B^{-\sigma}_{p,1}
\hookrightarrow \dot B^{-\sigma-\frac{3}{2p}}_{2p,1},
\quad \forall \sigma\in ]0,1[.
$$
Hence by the characterisation of the Besov spaces with negative index, we know
$\forall \sigma\in ]0,1[$,
\begin{align*}
\|t^{\frac 12+\frac \sigma2-\frac 1r} e^{t\Delta}\nabla^2 v_0\|_{L^r(\R^+;L^p)}&
+\|t^{\frac \sigma2-\frac 1s} e^{t\Delta}\nabla v_0\|_{L^s(\R^+;L^p)}
+\|t^{\frac \sigma2+\frac{3}{4p}-\frac 1q}e^{t\Delta}\nabla  v_0\|_{L^q(\R^+;L^{2p})}\\
&=\|\nabla^2 v_0\|_{\dot B^{-1-\sigma}_{p,r}}
+\|\nabla v_0\|_{\dot B^{-\sigma}_{p,s}}
+\|\nabla  v_0\|_{\dot B^{-\sigma-\frac{3}{2p}}_{2p,q}}
\leq C\| v_0\|_{W^{1,p}}.
\end{align*}
Using this with $\s=\frac1r-\frac12$ and $\frac1r$ respectively, we deduce
$$\|\D v_L\|_{L^r(\R^+;L^p)}\leq \|t^{\frac 12(\frac12-\frac 1r)} e^{t\Delta}\nabla^2 v_0\|_{L^r((0,1);L^p)}
+\|t^{\frac 12(1-\frac 1r)} e^{t\Delta}\nabla^2 v_0\|_{L^r((1,\infty);L^p)}\leq C\| v_0\|_{W^{1,p}}.$$
And the other two terms in the left hand side of \eqref{3.1} can be estimated similarly.
\end{proof}

\begin{lem}\label{lem3.2}
{\sl For any $p\in \bigl[\frac 32,\infty\bigr[,~r\in ]1,2[$,
and $q$ given by $\f1q=\f1r- \bigl(\frac 12-\f{3}{4p}\bigr)$,
there exists some positive constant $C$ such that
\begin{equation}\label{3.2}
\begin{split}
\Bigl\| \int^t_0 \Delta e^{(t-t')\Delta}& f(t')\,dt'
\Bigr\|_{L^r_T (L^p)}+
\Bigl\| \int^t_0 \nabla e^{(t-t')\Delta} f(t')\,dt'
\Bigr\|_{L^{\frac{2r}{2-r}}_T (L^p)}\\
&+\Bigl\| \int^t_0 \nabla
e^{(t-t')\Delta} f(t')\,dt' \Bigr\|_{L^{q}_T (L^{2p})}
\leq C\|f\|_{L^r_T(L^p)},
\quad\forall T\in \R^+,
\end{split}
\end{equation}}
\end{lem}
\begin{proof}
To prove \eqref{3.2} it suffices to write
$$
e^{(t-t')\Delta} f(t')=\frac{1}{\sqrt{t-t'}^3}K(\frac{\cdot}{\sqrt{t-t'}})\ast f(t', \cdot),
$$
with $K$ denoting the inverse Fourier transform of $e^{-|\xi|^2}$,
and then to take Young's inequality first in the space variable and  then in the time variable.
\end{proof}

Now we are in a position to prove $J_0(t)\leq C_0$.
As we are considering
perturbations of the reference density $1,$ it is convenient to set
$a\eqdefa 1/\rho-1$ so that System \eqref{1.1} translates into
\begin{equation}\label{2.1}
 \quad\left\{\begin{array}{lll}
\displaystyle \pa_t a + v \cdot \grad a=0 \qquad (t,x)\in\R^+\times\R^3,  \\
\displaystyle \pa_t v + v \cdot \grad v+ (1+a)(\grad\pi-\D v)=0, \\
\displaystyle \dive\, v = 0, \\
 \displaystyle (a, v)|_{t=0}=(a_0, v_{0}).
\end{array}\right.
\end{equation}

For any $\la>0$, and any function $g(t)$,  we denote
\begin{equation*} \label{3.5}
g_\la(t)\eqdefa
g(t)\exp\bigl(-\la\int_0^t V(t')\,dt'\bigr),
\quad \hbox{ with }V(t)\eqdefa
\|v(t)\|_{L^{2p}}^{\f{4p}{2p-3}}\geq 0.
\end{equation*}
Then by virtue of \eqref{2.1}, $v_\la$ satisfies
\begin{equation}\label{vlambda}
\d_t v_\lambda+ \lambda V(t)v_\lambda=(\d_t v)_\lambda
=\Delta v_\lambda
+(-v\cdot\nabla
v_\la +a\Delta v_\la-(1+a)\nabla\pi_\la),
\end{equation}
which can also be written in the form
\beq \label{3.6}
v_\la(t)=e^{t\Delta}v_{0,\lambda} +\int^t_0
e^{-\la\int_{t'}^t V(t'')\,dt''} e^{(t-t')\Delta}
\Bigl( -v\cdot\nabla
v_\la +a\Delta v_\la-(1+a)\nabla\pi_\la \Bigr)(t')\,dt'.
\eeq
Taking space divergence to  \eqref{vlambda} gives
 $$
 \D\pi_\la=-\dive(v\cdot\na
v_\la)+\dive\bigl(a(\D v_\la-\na\pi_\la)\bigr),
$$
from which and the fact that
$$\|a\|_{L^\infty}\leq\|a_0\|_{L^\infty}=\frac{|\eta|}{1-|\eta|}$$
since $a$ satisfies a free transport equation in \eqref{2.1}, we infer
\beq\label{3.7}
\|\na\pi_\la(t)\|_{L^p}\leq C\bigl(\|v\cdot\na
v_\la(t)\|_{L^p}+|\eta|\|\D v_\la(t)\|_{L^p}\bigr).
 \eeq
In view of \eqref{3.6}, we get, by applying
Lemma \ref{lem3.1}, Lemma \ref{lem3.2} and \eqref{3.7}, that
\begin{equation*}\begin{split}
\|\D v_\la&\|_{L^{r_0}_t(L^p)}
+\|\na v_\la\|_{L^{\f{2r_0}{2-r_0}}_t(L^p)}
+\|\na v_\la\|_{L^{q_0}_t(L^{2p})}\\
\leq &\|\D v_L \|_{L^{r_0}_t(L^p)}
+\|\na v_L \|_{L^{\f{2r_0}{2-r_0}}_t(L^p)}
+\|\na v_L \|_{L^{q_0}_t(L^{2p})}\\
&+C\biggl(\int_0^te^{-\la r_0\int_{t'}^t V(t'')\,dt''}\Bigl(\|v\cdot\na
v_\la(t')\|_{L^p}^{r_0}+\|a(t')\|_{L^\infty}^{r_0}\|\D
v_\la(t')\|_{L^p}^{r_0}\\
&\qquad\qquad\qquad\qquad\qquad\qquad\qquad\quad+\bigl(1+\|a(t')\|_{L^\infty}^{r_0}\bigr)\|\na\pi(t')\|_{L^p}^{r_0}\Bigr)\,dt'\biggr)^{\f{1}{r_0}}\\
\leq& C\biggl(\|v_0\|_{W^{1,p}}+|\eta|\|\D
v_\la\|_{L^{r_0}_t(L^p)} +\Bigl(\int_0^te^{-\la
r_0\int_{t'}^t  V(t'')\,dt''}\|v\cdot\na
v_\la(t')\|_{L^p}^{r_0}\,dt'\Bigr)^{\f{1}{r_0}}\biggr),
\end{split}\end{equation*}
where $q_0$ is taken to be $\f{1}{q_0}=\f{1}{r_0}- \bigl(\frac 12-\f{3}{4p}\bigr)$.
Thus we can use H\"older's inequality to get
\beno
\begin{split}
\Bigl(\int_0^t&e^{-\la r_0\int_{t'}^t V(t'')\,dt''}\|v\cdot\na
v_\la(t')\|_{L^p}^{r_0}\,dt'\Bigr)^{\f{1}{r_0}}\\
\leq
&\Bigl(\int_0^te^{-\f{4p }{2p-3}\lambda
\int_{t'}^t V(t'')\,dt''}\|v(t')\|_{L^{2p}}^{\f{4p}{2p-3}}\,dt'\Bigr)^{\f{2p-3}{4p}}\|\na
v_\la\|_{L^{q_0}_t(L^{2p})}\\
\leq &C \la ^{-\f{2p-3}{4p}}\|\na v_\la\|_{L^{q_0}_t(L^{2p})}.
\end{split}
\eeno
We take $|\eta|$ small enough and $\lambda$ large enough  to achieve
\beq \label {3.9}
\|\D v_\la\|_{L^{r_0}_t(L^p)}
+\|\na
v_\la\|_{L^{\f{2r_0}{2-r_0}}_t(L^p)}
+\|\na
v_\la\|_{L^{q_0}_t(L^{2p})}
\leq C\|v_0\|_{W^{1,p}} .
\eeq
By use of  the estimates \eqref{vLr} and \eqref{vL3},
we  deduce from the interpolation inequality that
\begin{equation}\label{3.12}
\bigl(\int^t_0 V(t')\,dt'\bigr)^{\frac{2p-3}{4p}}
=\|v\|_{L_t^{\f{4p}{2p-3}}(L^{2p})}
\leq
C\|v\|_{L_t^{\infty}(L^3)}^{1-\theta}\|\Delta v\|_{L_t^r(L^{\f{3r}{3r-2}})}^\theta
\leq C_0,
\end{equation}
where $\theta=\f{(2p-3)r}{4p}\in]0,1[$.
Thus we deduce from \eqref{3.9} that
\beq\label{3.13}
\begin{split}
\|\D v\|_{L^{r_0}_t(L^p)}
+&\|\na
v\|_{L^{\f{2r_0}{2-r_0}}_t(L^p)}+\|\na v\|_{L^{q_0}_t(L^{2p})}\\
\leq& C\|v_0\|_{W^{1,p}}\cdot e^{\la \int_0^t V(t')\,dt'}\leq C\|v_0\|_{W^{1,p}}\cdot e^{\la
C_0}\leq C_0.
\end{split} \eeq
This together with \eqref{3.7} and \eqref{3.12} ensures that
\beq\label{3.14}
\begin{split} \|\na\pi\|_{L^{r_0}_t(L^p)}\leq
C\Bigl(\|v\|_{L^{\f{4p}{2p-3}}_t(L^{2p})}\|\na
v\|_{L^{q_0}_t(L^{2p})}+|\eta|\|\D v\|_{L^{r_0}_t(L^p)}\Bigr)\leq C_0,
\end{split}
\eeq
and hence we deduce from the velocity equation in \eqref{2.1} that
\beq\label{3.14b}
\begin{split} \|\d_t v\|_{L^{r_0}_t(L^p)} \leq C_0.
\end{split}
\eeq

Moreover, for any $p>3,\ r_0\in]1,2[$, there holds
\beq\label{3.15}
\|\nabla v\|_{L_t^{\f{2pr_0}{2p+3r_0-pr_0}}(L^\infty)}\leq C\|\na
v\|_{L^{\f{2r_0}{2-r_0}}_t(L^p)}^{1-\f3p}\|\D v\|_{L^{r_0}_t(L^p)}^{\f3p}\leq C_0.
\eeq
It is easy to observe that when $r_0$ varies from 1 to 2, we have
$\sigma_1\eqdefa\f{2pr_0}{2p+3r_0-pr_0}\in\bigl]\f{2p}{p+3},\f{2p}{3} \bigr[$.
Similarly, we deduce that for $\sigma_2\eqdefa \f{2p-3}{p}\cdot\f{2r_0}{2-r_0} \in\bigl]\f{4p-6}{p},\infty \bigr[$, there holds
\begin{equation}\label{3.16}
\|v\|_{L_t^{\sigma_2}(L^\infty)}\leq
C\|v\|_{L_t^{\infty}(L^3)}^{\f{p-3}{2p-3}}\|\nabla v\|_{L_t^{\f{2r_0}{2-r_0}}(L^p)}^{\f{p}{2p-3}}\leq C_0.
\end{equation}

By view of the definition \eqref{J0}, we deduce $J_0(t)\leq C_0$ from the
estimates \eqref{vL3}, \eqref{3.13}, \eqref{3.14}, \eqref{3.14b}, \eqref{3.15}, \eqref{3.16},
which completes the proof of \eqref{bound} for $\ell=0$.

\subsection{The proof of \eqref{bound} for $\ell=1$.}\label{sub2.2}

We first deduce from the equation \eqref{W} and the proved fact $J_0(t)\leq C_0$ that
for any $p\in]3,\infty[$,
\begin{equation}\label{X:Lp}
\begin{split}
\| X_i(t)\|_{L^{p}\bigcap L^\infty}
\leq
\|X_{i,0}\|_{L^{p}\bigcap L^\infty}
\exp\Bigl(\int_0^t\|\na v(t')\|_{W^{1,p}}\,dt'\Bigr)
\leq \cH_1(t).
\end{split} \end{equation}
Then by applying $\pa_k,~k=1,2,3$ to equation \eqref{W}, we obtain
$$
\d_t \pa_k X_i+v\cdot \nabla \pa_k X_i+\d_k v\cdot\nabla X_i =\d_k
X_i\cdot\nabla v+X_i\cdot\na \d_k v.
$$
Then the standard energy estimates for transport equation leads to
\begin{equation}\label{X:W1p}\begin{split}
\|\na X_i(t)\|_{L^p}&\leq \Bigl(\|\na
X_{i,0}\|_{L^p}+\|X_i\|_{L^\infty_t(L^\infty)}\|\na^2
v\|_{L^1_t(L^p)}\Bigr)\exp\Bigl(2\int_0^t\|\na
v(t')\|_{L^\infty}\,dt'\Bigr)\\
& \leq  \cH_1(t).
\end{split}\end{equation}
Similarly, by applying the operator $\Delta$ to \eqref{W},
we get
$$
\d_t \Delta X_i+v\cdot\nabla\Delta X_i+2 \sum_{j=1}^3\p_jv \cdot
\nabla\p_jX_i
  +\Delta v\cdot\nabla X_i
=\Delta \d_{X_i} v.$$
Noting that $p\in]3,\infty[$, thus we can achieve, by using the standard $L^p$ energy estimate
and the interpolation inequality $\|f\|_{L^\infty}\leq C\|f\|_{L^p}^{1-\f3p}\|\na
f\|_{L^p}^{\f3p}\leq C\bigl(\|f\|_{L^p}+\|\na
f\|_{L^p}\bigr)$, that
\begin{align*}
\f{d}{dt}\|\D X_i(t)\|_{L^p}\leq & 2\|\na
v(t)\|_{L^\infty}\|\na^2X_i(t)\|_{L^p}+\|\D v(t)\|_{L^p}\|\na
X_i(t)\|_{L^\infty}+\|\D\p_{X_i}v(t)\|_{L^p}\\
\leq &C\Bigl(\bigl(\|\na v(t)\|_{L^p}+\|\D v(t)\|_{L^p}\bigr)\|\D
X_i(t)\|_{L^p}\\
&\qquad+\|\D v(t)\|_{L^p}\bigl(\|\na X_i(t)\|_{L^p}+\|\D X_i(t)\|_{L^p}\bigr)
+\|\D\p_{X_i}v(t)\|_{L^p}\Bigr).
\end{align*}
Applying Gronwall's inequality, together with \eqref{X:W1p} and $J_0(t)\leq C_0$, leads to
\begin{equation}\label{X:W2p}\begin{split}
\|\D X_i(t)\|_{L^p}\leq &C\Bigl(\|\D
X_{i,0}\|_{L^p}+ \|\D v\|_{L^1_t(L^p)}\|\na X_i\|_{L^\infty_t(L^p)}+\|\D
\p_{X_i}v\|_{L^1_t(L^p)}\Bigr)\\
&\quad\times\exp\Bigl(C\bigl(\|\na
v\|_{L^1_t(L^p)}+\|\D v\|_{L^1_t(L^p)}\bigr)\Bigr)\\
\leq & \cH_1(t)\bigl(1+\|\D \p_{X_i}v\|_{L^1_t(L^p)}\bigr).
\end{split}\end{equation}

Next we shall focus on the estimate of $\|\D \p_{X_i}v\|_{L^1_t(L^p)}$.
To do this, we first apply the operator $\d_{X_i}$ to the velocity equation in \eqref{2.1} to get the equation for $\d_{X_i} v$:
\begin{equation}\begin{split}\label{4.13}
\d_t (\d_{X_i}& v)
+v\cdot\nabla(\d_{X_i} v)
+(1+a)(\nabla\d_{X_i} \pi-\Delta \d_{X_i} v)=F_1\bigl(v,\pi,(i)\bigr),\quad \mbox{with}\\
&F_1\bigl(v,\pi,(i)\bigr)\eqdefa
(1+a)\bigl(\nabla X_i\cdot\nabla \pi-\Delta X_i\cdot\nabla v
-2\nabla X_i:\nabla^2 v\bigr),
\end{split}\end{equation}

A direct calculation shows that $(\d_{X_i} \pi)$ satisfies (noticing that $\div X_{i}=0$)
\begin{align*}
\div\bigl((1+a)&\nabla\d_{X_i} \pi\bigr)
=\div\Bigl(-\d_t(\d_{X_i} v) -v\cdot\nabla (\d_{X_i} v)
+\Delta\d_{X_i} v +a\Delta \d_{X_i} v+ F_1\Bigr)\\
&=\div\Bigl( -\d_t X_i\cdot\nabla v-\d_t v\cdot\nabla X_i-v\cdot\nabla(\d_{X_i} v)
+\Delta X_i\cdot\nabla v+2\nabla X_i:\nabla^2 v\\
&\qquad
+\Delta v\cdot\nabla X_i+a\Delta \d_{X_i} v+F_1(v,\pi,(i))\Bigr)\eqdefa\div G,
\end{align*}
which implies
$$\nabla\d_{X_i} \pi=\nabla\D^{-1}\div\bigl(G-a\nabla\d_{X_i} \pi\bigr).$$
Then by using the fact that Riesz transform is $L^p,~1<p<\infty$ bounded, we obtain
\begin{equation}\label{3.30}
(1-C\|a_0\|_{L^\infty})\|\nabla\d_{X_i} \pi\|_{L^p}\leq C\|G\|_{L^p}.
\end{equation}
For any $r_1\in \bigl]1, \f{2k}{k+\e}\bigr[,$  we can find some
$r_0\in]r_1,2[$ and
$\sigma_1\in
\bigl]r_1, \frac {2p}{3}\bigr[$.
Then we can use the equations \eqref{W},
together with the estimates \eqref{X:Lp},~\eqref{X:W1p},~$J_0(t)\leq C_0$, to get
\begin{equation}\begin{split}\label{3.31}
&\Bigl\|\Bigl(\d_t X_i\cdot\nabla v,\, v\cdot\nabla(\d_{X_i}v)\Bigr)\Bigr\|_{L^{r_1}_t(L^p)}
+\Bigl\|\Bigl(\d_t v\cdot\nabla X_i,\,
\nabla X_i:\nabla^2 v\Bigr)\Bigr\|_{L^{r_1}_t(L^p)}\\
&\lesssim\|\nabla v\|_{L^{\s_{1}}_t(L^\infty)}
\Bigl( \|X_i\|_{L^{\infty}_t(L^\infty)}\|\nabla v\|_{L^{\frac{\s_1 r_1}{\s_1-r_1}}_t(L^p)}
+\|\nabla X_i\|_{L^{\infty}_t(L^p)}\|v\|_{L^{\frac{\s_1 r_1}{\s_1-r_1}}_t(L^\infty)}\Bigr) \\
&\qquad +\|X_i\|_{L^\infty_t(L^\infty)}
 \|\nabla^2 v\|_{L^{r_{0}}_t(L^p)}
\|v\|_{L^{\frac{r_0 r_1}{r_0-r_1}}_t(L^\infty)}
+\bigl\|\bigl(\pa_t v,\,\D v\bigr)\bigr\|_{L^{r_1}_t(L^p)}\|\nabla X_i\|_{L^\infty_t(L^\infty)}\\
&\lesssim \cH_1(t)\bigl(1+\|\nabla X_i\|_{L^\infty_t(L^\infty)}\bigr).
\end{split}\end{equation}
And Similarly we have
\begin{equation}\begin{split}\label{3.32}
\|F_1\|_{L^{r_1}_t(L^p)}&\leq\|\nabla{X_i}\|_{L^\infty_t(L^\infty)}\bigl(\|\D
v\|_{L^{r_1}_t(L^p)}+\|\na\pi\|_{L^{r_1}_t(L^p)}\bigr)
+\|\Delta X_i\cdot\nabla v\|_{L^{r_1}_t(L^p)}\\
&\leq C_0\|\na
X_i\|_{L^\infty_t(L^\infty)}+\Bigl(\int_0^t\|\na
v(t')\|_{L^\infty}^{r_1}\|\D
X_i(t')\|_{L^p}^{r_1}\,dt'\Bigr)^{\f1{r_1}}.
\end{split}\end{equation}

Then we achieve, by substituting \eqref{3.31} and \eqref{3.32} into \eqref{3.30},
and choosing $\|a_0\|_{L^\infty}$ to be sufficiently small, that
\begin{equation}\begin{split}\label{3.33}
\|\nabla\d_{X_i} \pi\|_{L^{r_1}_t(L^p)}\leq
\cH_1(t)\bigl(1+\|\nabla X_i&\|_{L^\infty_t(L^\infty)}\bigr)
+\|a_0\|_{L^\infty}\|\D\pa_{X_i}v\|_{L^{r_1}_t(L^p)}\\
&+\Bigl(\int_0^t\|\na
v(t')\|_{L^\infty}^{r_1}\|\D X_i(t')\|_{L^p}^{r_1}\,dt'\Bigr)^{\f1{r_1}}.
\end{split}\end{equation}

On the other hand, we can get from \eqref{4.13} that
\begin{equation}\begin{split}\label{3.34}
&(\d_{X_i} v)(t)=
e^{t\Delta}(\d_{X_{i,0}} v_0)\\
&\quad+\int^t_0 e^{(t-t')\Delta}
\Bigl( -v\cdot\nabla(\d_{X_i} v)
+a\Delta \d_{X_i} v
-(1+a)\nabla\d_{X_i} \pi+F_1\bigl(v,\pi,(i)\bigr) \Bigr)(t')\,dt'.
\end{split}\end{equation}
Similarly as Lemma \ref{lem3.1}, we can prove the following estimate:
\begin{equation}\label{3.35}
\bigl\|\D e^{t\D}\p_{X_{i,0}}^\ell v_0\bigr\|_{L^{r_\ell}_t(L^p)}
+\bigl\|\na e^{t\D}\p_{X_{i,0}}^\ell v_0\bigr\|_{L^{s_\ell}_t(L^p)}\leq C\|\p_{X_{i,0}}^\ell
v_0\|_{W^{1-\f\ell{k}\e,p}},
\quad \forall \ell=1,\cdots,k.
\end{equation}
Using this, we infer from \eqref{3.34} that
\begin{equation*}\begin{split}
\|\p_t\p_{X_i} v\|_{L^{r_1}_t(L^p)}+&\|\D\p_{X_i} v\|_{L^{r_1}_t(L^p)}
\leq
\|\p_{X_{i,0}}v_0\|_{W^{1-\f\e{k},p}}+C\Bigl(\|v\cdot\na(\p_{X_i}
v)\|_{L^{r_1}_t(L^p)}\\
&+\|a_0\|_{L^\infty}\|\D\p_{X_i}v\|_{L^{r_1}_t(L^p)}+
\|\na\p_{X_i}\pi\|_{L^{r_1}_t(L^p)}+\|F_1\|_{L^{r_1}_t(L^p)}\Bigr).\\
\end{split}\end{equation*}
Hence by virtue of \eqref{3.31}-\eqref{3.33}, we get, by taking $\|a_0\|_{L^\infty}$
to be sufficiently small, that
\begin{equation}\begin{split}\label{3.36}
\|\p_t\p_{X_i}& v\|_{L^{r_1}_t(L^p)}+\|\D\p_{X_i} v\|_{L^{r_1}_t(L^p)}\\
&\leq \cH_1(t)\bigl(1+\|\na
X_i\|_{L^\infty_t(L^\infty)}\bigr)
+C\Bigl(\int_0^t\|\na v(t')\|_{L^\infty}^{r_1}\|\D
X_i(t')\|_{L^p}^{r_1}\,dt'\Bigr)^{\f1{r_1}}.
\end{split}\end{equation}
Substituting the above estimate into \eqref{X:W2p},
and using \eqref{X:W1p}, gives rise to
\begin{equation*}\begin{split}
\|\D X_i\|_{L^\infty_t(L^p)}
&\leq\cH_1(t)\Bigl(1+\|\na X_i\|_{L^\infty_t(L^p)}^{1-\frac2p}
\|\D X_i\|_{L^\infty_t(L^p)}^{\frac2p}
+\bigl(\int_0^t\|\na v(t')\|_{L^\infty}^{r_1}
\|\D X_i(t')\|_{L^p}^{r_1}\,dt'\bigr)^{\f1{r_1}}\Bigr)\\
&\leq \cH_1(t)+\frac12\|\D X_i\|_{L^\infty_t(L^p)}
+\cH_1(t)\Bigl(\int_0^t\|\na v(t')\|_{L^\infty}^{r_1}
\|\D X_i(t')\|_{L^p}^{r_1}\,dt'\Bigr)^{\f1{r_1}},
\end{split}\end{equation*}
which gives
$$\|\D X_i\|_{L^\infty_t(L^p)}^{r_1}\leq
\cH_1(t)
+\cH_1(t)\int_0^t\|\na v(t')\|_{L^\infty}^{r_1}
\|\D X_i(t')\|_{L^p}^{r_1}\,dt'.$$
Then using Gronwall's inequality, together with the fact $J_0(t)\leq C_0$,
 we obtain
\begin{equation}
\|\D X_i\|_{L^\infty_t(L^p)}^{r_1}\leq\cH_1(t)
\exp\Bigl(\cH_1(t)\int_0^t\|\na v(t')\|_{L^\infty}^{r_1}\,dt'\Bigr)
\leq\cH_2(t).
\end{equation}
This together with \eqref{3.33} and \eqref{3.36} leads to
\begin{equation}
\|\nabla\d_{X_i} \pi\|_{L^{r_1}_t(L^p)}
+\|\p_t\p_{X_i}v\|_{L^{r_1}_t(L^p)}+\|\D\p_{X_i} v\|_{L^{r_1}_t(L^p)}
\leq\cH_2(t).
\end{equation}

By now, we have completed the proof of \eqref{bound} for $\ell=1$.

\subsection{The proof of \eqref{bound} for $\ell\geq 2$.}\label{sub2.3}
\
First, we would like to point out that, the Section $6$ of \cite{LZ}
gives the corresponding proof for 2-D case,
but the estimates there are in fact independent of spatial dimension.
Thus all we need to do here is to write clear the
differences in notations, as
we need to consider taking $\ell(\geq 2)$ order derivatives in
different tangential directions.

Before proceeding, we state the following commutator estimates which can be
viewed as generalizations of Lemma $6.1$ and Remark $6.1$ in \cite{LZ}.
\begin{lem}\label{lem4.1}
{\sl For any $\ell\in\{1,\ldots,k\}$ and any system of vector fields $\cW=\{W_1, \cdots, W_N\}$,
 let $\alpha(\ell)=(\alpha_1,\ldots,\alpha_\ell)$ be a
multi-index of length $\ell$ with
 indices taking value in $\{1,\ldots,N\}$,
and we denote
$\widehat\alpha(i)=(\alpha_{\ell-i+1},\ldots,\alpha_\ell)$
for $i=0,\cdots,\ell$.
Then  there exists a positive constant $C$ such that $\forall\,r_\ell \in
\bigl]1,\f{2k}{k+\ell\e}\bigr[$, $s_\ell\in\bigl]2,\f{2k}{\ell\e}\bigr[$, $\forall\,X\in \cW$,
$\forall 1\leq i\leq\ell$,
\begin{equation}\label{4.3}\begin{split}
&\bigl\|\d_{\cW}^{\alpha(i)}\nabla \d_{\cW}^{\widehat\alpha(\ell-i)} X-\nabla \d_{\cW}^{\alpha(\ell)}
X\|_{L^\infty_t(W^{1,p})}+\bigl\|\d_{\cW}^{\alpha(i)}\nabla^2\d_{\cW}^{\widehat\alpha(\ell-i)} X-\nabla^2 \d_{\cW}^{\alpha(\ell)}
X\bigr\|_{L^\infty_t(L^p)}\\
&\qquad\qquad+\bigl\|\d_{\cW}^{\alpha(i)}\d_t\d_{\cW}^{\widehat\alpha(\ell-i)}
X-\d_t\d_{\cW}^{\alpha(\ell)} X\bigr\|_{L^{s_\ell}_t(W^{1,p})} \leq
CJ_{\ell}^{\ell+1},
\end{split} \end{equation}
and when $i\neq \ell$, there holds
\begin{equation} \label{4.4}\begin{split}
&\bigl\|\d_{\cW}^{\alpha(i)} \nabla \d_{\cW}^{\widehat\alpha(\ell-i)} v-\nabla\d_{\cW}^{\alpha(\ell)} v\bigr\|
_{L^{r_{\ell-1}}_t(L^\infty)\cap
L^{s_{\ell-1}}_t(L^p)}\\
&+\bigl\|\d_{\cW}^{\alpha(i)} \nabla^2 \d_{\cW}^{\widehat\alpha(\ell-i)}
v-\nabla^2\d_{\cW}^{\alpha(\ell)} v\bigr\| _{L^{r_{\ell-1}}_t(L^p)}
+\bigl\|\d_{\cW}^{\alpha(i)} \nabla \d_{\cW}^{\widehat\alpha(\ell-i)} \pi-\nabla\d_{\cW}^{\alpha(\ell)}
\pi\bigr\|_{L^{r_{\ell-1}}_t(L^p) }\\
&+\bigl\|\d_{\cW}^{\alpha(i)} \d_t\d_{\cW}^{\widehat\alpha(\ell-i)}
v-\d_t\d_{\cW}^{\alpha(\ell)} v\| _{L^{r_{\ell-1}}_t(L^p) } \leq C
J_{\ell-1}^{\ell+1} ,
\end{split} \end{equation}
when $i=\ell$, there holds
\begin{equation}\label{4.5}\begin{split}
&\bigl\|\d_{\cW}^{\alpha(\ell)}\nabla v-\nabla\d_{\cW}^{\alpha(\ell)}
v +\d_{\cW}^{\alpha(\ell-1)}\nabla X_{\alpha_\ell}\cdot\nabla v
\bigr\|_{L^{r_{\ell-1}}_t(L^\infty)\cap L^{s_{\ell-1}}_t(L^p)}
\leq C J_{\ell-1}^{\ell+1},
\\
&\bigl\|\d_{\cW}^{\alpha(\ell)} \Delta   v - \Delta\d_{\cW}^{\alpha(\ell)} v +\d_{\cW}^{\alpha(\ell-1)}
\Delta X_{\alpha_\ell}\cdot\nabla v +2\d_{\cW}^{\alpha(\ell-1)}\nabla X_{\alpha_\ell}:\nabla^2
v\bigr\|_{L^{r_{\ell-1}}_t(L^p)} \leq C J_{\ell-1}^{\ell+1},
\\
&\bigl\|\d_{\cW}^{\alpha(\ell)} \nabla   \pi-\nabla\d_{\cW}^{\alpha(\ell)} \pi
+\d_{\cW}^{\alpha(\ell-1)}\nabla X_{\alpha_\ell}\cdot\nabla\pi\bigr\| _{L^{r_{\ell-1}}_t(L^p)
} \leq C J_{\ell-1}^{\ell+1},\\
&\bigl\|\d_{\cW}^{\alpha(\ell)} \d_t v-\d_t\d_{\cW}^{\alpha(\ell)} v +\d_{\cW}^{\alpha(\ell-1)}\d_t
X_{\alpha_\ell}\cdot\nabla v\bigr\| _{L^{r_{\ell-1}}_t(L^p) } \leq C
J_{\ell-1}^{\ell+1}.
\end{split} \end{equation}
Moreover, it follows  from \eqref{Jell}, \eqref{4.3}, \eqref{4.4} and \eqref{4.5} that
for any $0\leq i\leq\ell$,
\begin{equation}\label{4.6}\begin{split}
&\|\d_{\cW}^{\alpha(i)}\nabla \d_{\cW}^{\widehat\alpha(\ell-i)} X \|_{L^\infty_t(W^{1,p})}
+\|\d_{\cW}^{\alpha(i)}\nabla^2\d_{\cW}^{\widehat\alpha(\ell-i)} X \|_{L^\infty_t(L^p)}\\
&\qquad\qquad+\|\d_{\cW}^{\alpha(i)}\d_t\d_{\cW}^{\widehat\alpha(\ell-i)} X \|_{L^{s_{\ell+1}}_t(W^{1,p})} \leq C
J_{\ell+1}^{\ell+1},\quad\mbox{and}
\end{split} \end{equation}

\begin{equation}\label{4.7}
\begin{split}
&\|\d_{\cW}^{\alpha(i)} \nabla \d_{\cW}^{\widehat\alpha(\ell-i)} v \| _{L^{r_{\ell }}_t(L^\infty)\cap
L^{s_{\ell}}_t(L^p)} +\|\d_{\cW}^{\alpha(i)} \nabla^2 \d_{\cW}^{\widehat\alpha(\ell-i)} v
\|_{L^{r_{\ell }}_t(L^p)}
\\
&\qquad\qquad\qquad+\|\d_{\cW}^{\alpha(i)} \nabla \d_{\cW}^{\widehat\alpha(\ell-i)} \pi \|
_{L^{r_{\ell }}_t(L^p) }+\|\d_{\cW}^{\alpha(i)} \d_t\d_{\cW}^{\widehat\alpha(\ell-i)} v \|
_{L^{r_{\ell }}_t(L^p) } \leq C J_{\ell}^{\ell+1}.
\end{split} \end{equation}
}
\end{lem}

\begin{proof}
Firstly, for any $X,Y\in {\cW}$, it is easy to observe that
\beno
\begin{split}
\|\p_Y\na
X-\na\p_YX&\|_{L^\infty_t(W^{1,p})}+\|\p_Y\na^2X-\na^2\p_YX\|_{L^\infty_t(L^p)}
+\|\p_Y\p_tX-\p_t\p_YX\|_{L^{s_1}_t(W^{1,p})}\\
&\qquad\quad\leq C\|\na X\|_{L^\infty_t(W^{1,p})}\bigl(\|\na
Y\|_{L^\infty_t(W^{1,p})}+\|\p_tY\|_{L^{s_1}_t(W^{1,p})}\bigr)\leq
CJ_1^2.
\end{split}
\eeno
This shows that \eqref{4.3} holds for $\ell=1.$
It is also easy to see that
\eqref{4.4} and \eqref{4.5} hold trivially for $\ell=1.$
Hence Lemma \ref{lem4.1} holds for $k=1.$

Let us now assume that \eqref{4.3}-\eqref{4.7} hold for
$\ell\leq j-1$ with $j\leq k.$ We are going to prove that
\eqref{4.3}-\eqref{4.5} also hold for $\ell=j,$  which will
imply immediately \eqref{4.6}-\eqref{4.7} for $\ell=j.$
In the following,
 for $n\leq m$ and for the $m$-length multi-index $(l_1,\cdots,l_m)$ such that
 $$
 (l_1,\cdots,l_m)\in L^n_m\eqdefa\bigl\{(l_1,\cdots,l_m)\,\big|\,
 l_1<\cdots<l_n,\,l_{n+1}<\cdots<l_{m},\, \{l_1,\cdots,l_m\}=\{1,\cdots,m\}\bigr\},
 $$
  we denote $\alpha^l(n)=(\alpha_{l_1}, \cdots, \alpha_{l_n})$
 and $\widehat\alpha^l(m-n)=(\alpha_{l_{n+1}}, \cdots, \alpha_{l_m})$.

For any positive integer $i\leq j-1$, a direct calculation gives
\begin{equation*}\begin{split}
\d_{\cW}^{\alpha(i+1)}\nabla\d_{\cW}^{\widehat\alpha(j-i-1)}
X-\nabla\d_{\cW}^{\alpha(j)} X&=\sum_{m=0}^i\d_{\cW}^{\alpha(m)}
[\pa_{X_{\alpha_{m+1}}},\nabla]
\d_{\cW}^{\widehat\alpha(j-m-1)}X\\
&=\sum_{m=0}^i\d_{\cW}^{\alpha(m)}\bigl(\nabla X_{\alpha_{m+1}}\cdot\nabla\d_{\cW}^{\widehat\alpha(j-m-1)}
X\bigr),
\end{split}\end{equation*}
where $[\cdot,\cdot]$ stands for the standard commutator.
Then the induction assumptions give
\begin{equation}\label{4.8}
\begin{split}
&\bigl\|\d_{\cW}^{\alpha(i+1)}\nabla\d_{\cW}^{\widehat\alpha(j-i-1)}
X-\nabla\d_{\cW}^{\alpha(j)} X\bigr\|_{L^\infty_t(W^{1,p})}\\
&\leq
C\sum_{m=0}^{i}\sum_{n=0}^{m}\sum_{(l_1,\cdots,l_m)\in L^n_m}
 \|\d_{\cW}^{\alpha^l(n)}\nabla X_{\alpha_{m+1}}\|_{L^\infty_t(W^{1,p})}
\|\pa_{\cW}^{\widehat\alpha^l(m-n)}\nabla\d_{\cW}^{\widehat\alpha(j-m-1)}
X\|_{L^\infty_t(W^{1,p})}
\\
&\leq C\sum_{m=0}^{i}\sum_{n=0}^{m}
 J_{n+1}^{n+1}  J_{j-n }^{j-n}
\leq CJ_j^{j+1}.
\end{split} \end{equation}

We follow the same lines as above to obtain
\begin{equation*}\begin{split}
\bigl\|\d_{\cW}^{\alpha(i+1)}&\nabla^2\d_{\cW}^{\widehat\alpha(j-i-1)} X-\nabla^2 \d_{\cW}^{\alpha(j)}
X\bigr\|_{L^\infty_t(L^p)}=\bigl\|\sum_{m=0}^i\d_{\cW}^{\alpha(m)}[\pa X_{\alpha_{m+1}},\nabla^2]\d_{\cW}^{\widehat\alpha(j-m-1)}
X\bigr\|_{L^\infty_t(W^{1,p})}\\
&\leq  C\sum_{m=0}^{i}\sum_{n=0}^{m} \sum_{(l_1,\cdots,l_m)\in L^n_m}
\Bigl( \|\d_{\cW}^{\alpha^l(n)}\nabla^2 X_{\alpha_{m+1}}\|_{L^\infty_t(L^{p})}
\|\pa_{\cW}^{\widehat\alpha^l(m-n)}\nabla\d_{\cW}^{\widehat\alpha(j-m-1)}X\|_{L^\infty_t(L^\infty)}
\\
&\qquad\qquad\qquad
+\|\d_{\cW}^{\alpha^l(n)}\nabla X_{\alpha_{m+1}}\|_{L^\infty_t(L^{\infty})}
\|\pa_{\cW}^{\widehat\alpha^l(m-n)}\nabla^2\d_{\cW}^{\widehat\alpha(j-m-1)}X\|_{L^\infty_t(L^p)}\Bigr)
\\
&\leq C \sum_{m=0}^{i}\sum_{n=0}^{m}
 J_{n+1}^{n+1}  J_{j-n }^{j-n}
\leq CJ_j^{j+1},\quad\mbox{and}
\end{split} \end{equation*}
$$
\bigl\|\d_{\cW}^{\alpha(i+1)}\pa_t\d_{\cW}^{\widehat\alpha(j-i-1)}
X-\pa_t\d_{\cW}^{\alpha(j)} X\bigr\|_{L^{s_j}_t(W^{1,p})}
\leq CJ_j^{j+1}.
$$

These two estimates together with \eqref{4.8} guarantee
 that \eqref{4.3} holds for $\ell=j$.

The same argument to achieve \eqref{4.3}  for $\ell=j$ yield \eqref{4.4} and \eqref{4.5} for $\ell=j$. We complete the proof of Lemma \ref{lem4.1} by the  induction argument.
\end{proof}

Next, we introduce the corresponding term to $F_\ell(v,\pi)$ given in Lemma $6.2$ of \cite{LZ}.
For any multi-index $\alpha(\ell)=(\alpha_1,\ldots,\alpha_{\ell})$,
we apply the operator $\pa_{\cX}^{\alpha(\ell-1)}$ to \eqref{4.13} for $\d_{X_{\alpha_\ell}}v$  to get
\begin{equation} \label{4.14}\begin{split}
&\d_t\d_{\cX}^{\alpha(\ell)} v + v\cdot\nabla \d_{\cX}^{\alpha(\ell)} v -  (1+a)\bigl(\Delta\d_{\cX}^{\alpha(\ell)} v -
\nabla \d_{\cX}^{\alpha(\ell)} \pi\bigr)\eqdefa F_\ell\bigl(v,\pi,\alpha(\ell)\bigr).
\end{split} \end{equation}
Here $F_\ell\bigl(v,\pi,\alpha(\ell)\bigr)$ is given by induction
\begin{equation*}\label{4.17}
F_\ell\bigl(v,\pi,\alpha(\ell)\bigr)=\pa_{X_{\alpha_1}}F_{\ell-1}\bigl(v,\pi,\widehat\alpha(\ell-1)\bigr)
+F_{1}\bigl(\pa_{\cX}^{\widehat\alpha(\ell-1)}v,\pa_{\cX}^{\widehat\alpha(\ell-1)}\pi,(\alpha_1)\bigr),
\end{equation*}
and hence by view of the definition of $F_1$ in \eqref{4.13}, and \eqref{bound:a}, we arrive at
\begin{equation*}\label{4.18}
\begin{split}
&F_\ell\bigl(v,\pi,\alpha(\ell)\bigr)
=\sum_{i=0}^{\ell-1}\pa_{\cX}^{\alpha(\ell-1-i)}
F_{1}\bigl(\pa_{\cX}^{\widehat\alpha(i)}v,\pa_{\cX}^{\widehat\alpha(i)}\pi,(\alpha_{\ell-i})\bigr)\\
&=(1+a)\sum_{i=0}^{\ell-1}\pa_{\cX}^{\alpha(\ell-1-i)}
\Bigl(  \nabla X_{\alpha_{\ell-i}}\cdot\nabla \pa_{\cX}^{\widehat\alpha(i)}\pi
- \Delta  X_{\alpha_{\ell-i}}\cdot\nabla \pa_{\cX}^{\widehat\alpha(i)}v
-2\nabla X_{\alpha_{\ell-i}}:\nabla^2 \pa_{\cX}^{\widehat\alpha(i)}v
 \Bigr).
\end{split}
\end{equation*}

Then exactly along the same line as the proofs of Lemma $6.2,~6.3$ and Proposition $6.1$ in \cite{LZ},
by using Lemma \ref{lem4.1} instead of Lemma $6.1$ and Remark $6.1$ used there,
we achieve a similar estimate as follows
\begin{align*}
&\|(\d_t \d_{\cX}^{\alpha(\ell)} v,\Delta\d_{\cX}^{\alpha(\ell)} v, \nabla\d_{\cX}^{\alpha(\ell)
}\pi)\|_{L^{r_{\ell }}_t(L^p)}
+\|\nabla\d_{\cX}^{\alpha(\ell)} v\|_{L^{r_{\ell }}_t(L^\infty) \cap
L^{s_{\ell }}_t(L^p)}+\|\d_{\cX}^{\alpha(\ell)} v\|_{L^{s_{\ell }}_t(L^\infty)\cap L^\infty_t(L^p)}\\
&\quad+ \|\d_t \d_{\cX}^{\alpha(\ell-1)}{\cX}\|_{L^{s_{\ell }}_t(W^{1,p})} +
\|\d_{\cX}^{\alpha(\ell-1)} {\cX}\|_{L^\infty_t(W^{2,p})}
\leq\cH_{\ell+1}(t),
\quad\forall \ell=2,\cdots,k,
\quad \forall t\in\R^+,
\end{align*}
which clearly implies \eqref{bound} for $\ell\geq2$ since the choice of $\alpha(\ell)$
is arbitrary.

\subsection{The proof of Proposition \ref{propapriori}}\label{sub2.4}

The existence of a weak solution $(\rho,v)$ to \eqref{1.1} is guaranteed by Proposition \ref{thmexistence}.
Moreover, the previous part of this section has shown that, this solution satisfies
the energy estimate \eqref{bound}, which obviously implies that $v\in L^1_\loc(\R^+, W^{2,p}(\R^3))$.
In particular, $v\in L^1_{\loc}(\R^+;\Lip)$,
thus the uniqueness of the solutions
can be proved by using Lagrangian formulation of \eqref{1.1} as  in \cite{DM12, HPZ}.
We omit the details here.

The fact that $v$ is in $L^1_{\loc}(\R^+;\Lip)$ also implies the existence and uniqueness
of the solution to \eqref{W} in $L^\infty_{\loc}(\R^+;W^{2,p})$, by using the classical
theory on transport equations. Then the energy estimate \eqref{bound} implies that,
this solution $X_i$ satisfies the conormal regularity in \eqref{claim}.
This completes the proof of Proposition \ref{propapriori}.

\smallskip
\noindent {\bf Acknowledgments.} This work was done when we were
visiting Morningside Center of the Academy of Mathematics and
Systems Science, CAS. We would like to thank Professor Ping Zhang
for introducing this interesting problem to us.
X. Liao was supported by SFB $1060$, University Bonn, during the last part of the work.

\end{document}